\documentclass[preprint,a4paper,11pt,sort]{elsarticle}
\usepackage[top=20truemm,bottom=25truemm,left=20truemm,right=20truemm]{geometry}

\usepackage{amsmath,amssymb}
\usepackage{amsthm} 
\usepackage{tikz}
\usepackage{bm}
\usepackage{tcolorbox}
\usepackage{mathtools}

\definecolor{darkred}{rgb}{0.75,0,0}

\usepackage[mathlines]{lineno}
\newcommand*\patchAmsMathEnvironmentForLineno[1]{
  \expandafter\let\csname old#1\expandafter\endcsname\csname #1\endcsname
  \expandafter\let\csname oldend#1\expandafter\endcsname\csname end#1\endcsname
  \renewenvironment{#1}
     {\linenomath\csname old#1\endcsname}
     {\csname oldend#1\endcsname\endlinenomath}}
\newcommand*\patchBothAmsMathEnvironmentsForLineno[1]{
  \patchAmsMathEnvironmentForLineno{#1}
  \patchAmsMathEnvironmentForLineno{#1*}}
\AtBeginDocument{
\patchBothAmsMathEnvironmentsForLineno{equation}
\patchBothAmsMathEnvironmentsForLineno{align}
\patchBothAmsMathEnvironmentsForLineno{flalign}
\patchBothAmsMathEnvironmentsForLineno{alignat}
\patchBothAmsMathEnvironmentsForLineno{gather}
\patchBothAmsMathEnvironmentsForLineno{multline}
}

\usepackage{thmtools} 
\usepackage[colorlinks=true]{hyperref}
\usepackage[capitalize,nameinlink]{cleveref}

\newtheorem{theorem}{Theorem}[section]

\newtheorem{lemma}[theorem]{Lemma}
\crefname{lemma}{Lemma}{Lemmas}

\newtheorem{claim}{Claim}[theorem]
\crefname{claim}{Claim}{Claims}

\crefname{corollary}{Corollary}{Corollaries}

\crefname{proposition}{Proposition}{Propositions}

\newtheorem{observation}[theorem]{Observation}
\crefname{observation}{Observation}{Observations}

\newenvironment{subproof}[1][\proofname]{%
  \begin{proof}[#1]%
}{%
  \end{proof}%
}

\newcommand{\proofsubparagraph}{\paragraph*}

\newcommand{\cp}{\mathbin{\square}} 
\newcommand{\tw}{\mathrm{tw}} 
\newcommand{\bn}{\mathrm{bn}} 

\newcommand{\dist}{\mathrm{dist}} 

\newcommand{\torus}[1]{T_{#1}} 
\newcommand{\Tn}{\torus{n}} 
\newcommand{\NTn}{N_{\Tn}} 

\newcommand{\Todd}{\torus{2m+1}} 
\newcommand{\NTodd}{N_{\Todd}} 

\newcommand{\Teven}{\torus{2m}} 
\newcommand{\NTeven}{N_{\Teven}} 

\newcommand{\Gi}{\Gamma_{\infty}} 
\newcommand{\NGi}{N_{\Gi}} 
\newcommand{\bGi}{\beta_{\infty}} 

\newcommand{\Z}{\mathbb{Z}} 

\begin{document}

\title{Treewidth of the $n \times n$ toroidal grid\tnoteref{t1}}
\tnotetext[t1]{%
Partially supported by JSPS KAKENHI Grant Numbers
JP22H00513, 
JP24H00697, 
JP24K23847, 
JP25K03076, 
JP25K03077, 
JP26KJ1299, 
and by JST SPRING Grant Number JPMJSP2125. 
} 

\author[hokkaido]{Tatsuya Gima}
\ead{gima@ist.hokudai.ac.jp}

\author[nagoya]{Hiraku Morimoto}
\ead{hiraku.morimoto@nagoya-u.jp}

\author[nagoya,jspsrf]{Yuto Okada}
\ead{research@yutookada.com}

\author[nagoya]{Yota Otachi}
\ead{otachi@nagoya-u.jp}

\affiliation[hokkaido]{
    organization={Hokkaido University},
    city={Sapporo},
    country={Japan}
}

\affiliation[nagoya]{
    organization={Nagoya University},
    city={Nagoya},
    country={Japan}
}

\affiliation[jspsrf]{
JSPS Research Fellow
}



\begin{abstract}

In this paper, we show that the treewidth of the $n \times n$ toroidal grid is $2n-1$ for all $n \ge 5$.
This closes the gap between the previously known upper bound of $2n-1$ (Ellis and Warren, DAM 2008) and
the lower bound of $2n-2$ (Kiyomi, Okamoto, and Otachi, DAM 2016).
To establish the matching lower bound, we construct a bramble of maximum order by utilizing maximum components obtained after removing $2n-1$ vertices. 
Our construction relies on the vertex-isoperimetric properties of the infinite grid to establish tight lower bounds on neighborhood sizes,
combined with a careful analysis of balls of radius $n/2-1$ and their boundaries to overcome structural obstructions when $n$ is even.

\par 
\end{abstract}


\maketitle


\section{Introduction}
\label{sec:intro}

For integers $m, n \ge 3$, let $\torus{m,n}$ be the $m \times n$ toroidal grid;
that is, $\torus{m,n}$ is the Cartesian product $C_{m} \cp C_{n}$ of cycles $C_{m}$ and $C_{n}$ of length $m$ and $n$, respectively.
As we primarily focus on the case of $m = n$, we adopt the shorthand $\Tn$ for $\torus{n,n}$.

In this paper, we study the treewidth of $\Tn$, denoted $\tw(\Tn)$.
Ellis and Warren~\cite{EllisW08} showed that the pathwidth of $\Tn$ is $2n-1$.
Since treewidth is at most pathwidth for every graph, their result implies that $\tw(\Tn) \le 2n-1$.
Wood~\cite{Wood13} showed that $\tw(\Tn) \ge 2n-5$.
Kiyomi et al.~\cite{KiyomiOO16} improved the lower bound by showing that $\tw(\Tn) \ge 2n-2$.
Combining these results, it follows that $\tw(\Tn) \in \{2n-2, 2n-1\}$.
Aidun et al.~\cite{AidunDMYY20} studied the treewidth of possibly non-square toroidal grids $\torus{n,n+c}$ 
and showed that for $c \ge 2$, $\tw(\torus{n,n+c}) = 2n$ 
and that $\tw(\torus{n,n+1}) \in \{2n-1,2n\}$.

Our goal in this paper is to determine $\tw(\Tn)$.
We present an improved lower bound of $\tw(\Tn) \ge 2n-1$ for $n \ge 5$, which matches the known upper bound mentioned above and gives the following theorem.
\begin{theorem}
\label{thm:tw}
$\tw(\Tn) = 2n-1$ for $n \ge 5$.
\end{theorem}
\noindent
For the remaining small cases, it is known that $\tw(\torus{3}) = 5$ ($= 2 \cdot 3 -1$) and $\tw(\torus{4}) = 6$ ($= 2\cdot 4 -2$)~\cite{KiyomiOO16}.
Thus, the treewidth of $\Tn$ is now determined for all $n \ge 3$.

\subsection{Treewidth and bramble number}

The \emph{treewidth} of a graph $G$, denoted $\tw(G)$, represents how close $G$ is to a tree.
It is defined using so-called \emph{tree decompositions}, which are often used to design efficient algorithms.
Here, we omit this definition as we do not use tree decompositions directly (see Bodlaender's survey~\cite{Bodlaender98} for the tree-decomposition based definition). 
Instead, we introduce a dual characterization based on brambles.

Let $G = (V,E)$ be a graph.
Two vertex subsets $S, S' \subseteq V$ \emph{touch} if $S \cap S' \ne \emptyset$ or there is an edge $\{v,v'\} \in E$ with $v \in S$ and $v' \in S'$.
A family $\mathcal{S} \subseteq 2^{V}$ of vertex subsets of $G$ is a \emph{bramble} of $G$
if every set in $\mathcal{S}$ is connected and any two elements in $\mathcal{S}$ touch.
The \emph{order} of a bramble is the size of a minimum hitting set of the bramble (i.e., a minimum set intersecting every element of the bramble).
The \emph{bramble number} of $G$, denoted $\bn(G)$, is the maximum order of a bramble of $G$.

Seymour and Thomas~\cite{SeymourT93} showed that $\tw(G) = \bn(G)-1$ for every graph $G$.
Therefore, proving \cref{thm:tw} is equivalent to showing that $\bn(\Tn) = 2n$ for $n \ge 5$.
Since we already know that $\bn(\Tn) \le 2n$ (equivalently $\tw(\Tn) \le 2n-1$) for all $n \ge 3$ by~\cite{EllisW08},
our goal is to show the following lower bound:
\begin{equation}
\label{eq:bramble_lb}
\bn(\Tn) \ge 2n \quad \text{for} \quad n \ge 5. 
\end{equation}

Brambles provide a constructive way to prove lower bounds on treewidth.
Indeed, brambles were used in all previous results showing lower bounds on the treewidth of toroidal grids~\cite{Wood13,KiyomiOO16,AidunDMYY20}.
We also use brambles and prove \eqref{eq:bramble_lb}.
See \cref{sec:overview} for an overview of the proof.

\subsection{Overview of the proof}
\label{sec:overview}

In the previous studies~\cite{Wood13,KiyomiOO16,AidunDMYY20}, the constructed brambles of $\Tn$ had common properties:
each member of the brambles is formed by a constant number of rows and columns possibly with some missing vertices. 
Using the property that each element contains almost full rows or almost full columns, it is possible to show that the families are brambles.
Also, using \emph{sparsity} of the elements and the carefully chosen \emph{holes} (i.e., missing vertices),
it can be shown that they do not admit small hitting sets.

In this paper, we take a different approach for designing a bramble.
It is based on the following simple necessary condition of an order-$k$ bramble $\mathcal{B}$ of a graph $G = (V,E)$:
for every $S \subseteq V$ with $|S| < k$, there is a connected set $B \in \mathcal{B}$ such that $B \cap S = \emptyset$.
To satisfy this condition with $G = \Tn$ and $k = 2n$, we construct a family $\mathcal{F}$ of vertex subsets as follows.
For each $J \subseteq V(\Tn)$ with $|J| = 2n-1$, we take one connected component $C$ of $\Tn - J$ and put $V(C)$ into $\mathcal{F}$.
When selecting $C$, we choose a maximum component, based on the intuition that larger sets are more likely to touch.
Clearly, the constructed family contains connected sets only and does not admit a hitting set of size smaller than $2n$.
If any two sets in the family touch, then we have a desired bramble and we are done.
Depending on the parity of $n$, we have two different cases.
\begin{itemize}
  \setlength{\itemsep}{0pt}
  
  \item For odd $n$, we show that the constructed family is indeed a bramble:
  each set in the family is large and has large neighborhood,
  and consequently any two sets in the family touch.

  \item For even $n$, the constructed family contains pairs of non-touching sets, and thus, it is not a bramble.
  However, we can characterize such pairs and handle them separately in the construction.
  We then show that the modified construction gives a bramble for even $n$.
\end{itemize}
In both cases, we often need to prove that certain vertex sets have large neighborhoods.
To this end, we use some connections between toroidal grids and the infinite grid.
For the infinite grid, for each size~$s$, the size-$s$ sets achieving the minimum neighborhood size are well understood~\cite{WangW77,AltshulerYVWB06,Sieben08,GuptaLMS21}.
By carefully mapping vertex sets of a toroidal grid to those of the infinite grid, we show the desired lower bounds on the size of neighborhoods in toroidal grids.

\section{Preliminaries}
\label{sec:preliminaries}

For integers $l$ and $u$ with $l \le u$, we set $[l, u] = \{i \in \Z \mid l \le i \le u\}$.
For a positive integer~$n$, we set $[n] = [1,n]$.
We define the addition over $[n]^{2}$ as $(i,j) + (i',j') = (i+i', j+j')$,
where both $i+i'$ and $j+j'$ are computed modulo $n$, with representatives $[n]$.

We represent the vertex set of $\Tn$ by $V(\Tn) = [n]^{2}$ and the edge set by
$E(\Tn) = \{\{v, v + (1,0)\}, \{v, v + (0,1)\} \mid v \in [n]^{2}\}$.
Following the matrix-like convention, we call the vertex set $\{i\} \times [n]$ the \emph{$i$th row} and $[n] \times \{j\}$ the \emph{$j$th column}.
See \cref{fig:torus}.

\begin{figure}[tbh]
  \centering
  \includegraphics{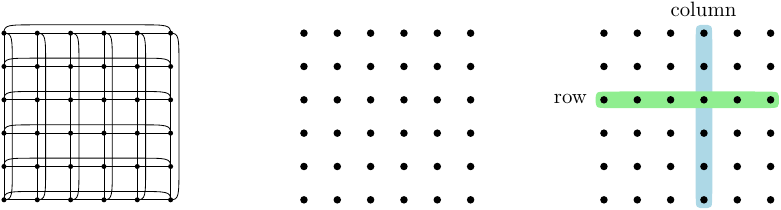}
  \caption{(Left) $\Tn$ with $n = 6$. (Center) We often draw $\Tn$ without edges. (Right) A row and a column.}
  \label{fig:torus}
\end{figure}

Let $G = (V,E)$ be a graph. 
For $S \subseteq V$, we denote by $G[S]$ the subgraph induced by $S$.
If $G[S]$ is connected, then we say that $S$ is \emph{connected}.
For $S \subseteq V$, let $G - S$ denote the graph $G[V \setminus S]$.
We call a connected component of $G$ with the maximum number of vertices a \emph{maximum component} of $G$. 
Note that a graph may have two or more maximum components.
For $v \in V$, let $N_{G}(v) = \{u \mid \{u,v\} \in E\}$ and $N_{G}[v] = \{v\} \cup N_{G}(v)$.
For $S \subseteq V$, let $N_{G}(S) = \left(\bigcup_{v \in S} N_{G}(v)\right) \setminus S$ and $N_{G}[S] = S \cup N_{G}(S)$.
By $\dist_{G}(u,v)$, we denote the \emph{distance} between $u$ and $v$ in $G$;
that is, $\dist_{G}(u,v)$ is the number of edges in a shortest $u$--$v$ path in $G$.

\section{Lower bounds on neighborhood size via infinite grids}
\label{sec:infinite}

In this section, we present lower bounds on the size $|\NTn(S)|$ for $S \subseteq V(\Tn)$ with certain conditions
in terms of the minimum size of the neighborhood of size-$|S|$ sets in the infinite grid.

By $\Gi$, we denote the \emph{infinite grid}; that is, $\Gi$ is the infinite graph such that the vertex set is $\Z^{2}$ and 
two vertices are adjacent in $\Gi$ if the $\ell_{1}$-distance between them is~$1$.
For a non-negative integer $s$, let $\bGi(s) = \min_{S \subseteq \Z^{2}, \, |S| = s} |N_{\Gi}(S)|$.

Wang and Wang~\cite{WangW77} presented a vertex ordering on $\Gi$ such that, for every~$s$,
the set $S$ of the first~$s$ elements in the ordering satisfies $\bGi(s) = |N_{\Gi}(S)|$.%
\footnote{Such an ordering is known also for $\Tn$ with even $n$~\cite{Karakhanyan82,Riordan98},
and actually, the ordering is the same as that for the infinite grid~\cite{BezrukovL09}.
Unfortunately, we cannot directly use these results because we deal with odd $n$ as well
and we need more properties than a single ordering as one can see in this section.}
Using this ordering,
Sieben~\cite[Theorem~5.3]{Sieben08} showed that $\bGi(s) \ge \lceil 2 + \sqrt{8s-4} \, \rceil$
and Altshuler et al.~\cite[Theorem~1]{AltshulerYVWB06} determined the exact value of $\bGi(s)$ for all $s$.%
\footnote{To be more precise, they studied the function $n(s) = \min_{s' \ge s} \bGi(s')$. We can see that $n(s) = \bGi(s)$ as $\bGi$ is non-decreasing.}
Gupta et al.~\cite{GuptaLMS21} characterized all vertex sets $S$ that satisfy $\bGi(|S|) = |N_{\Gi}(S)|$.
Their result showing that such $S$ is connected unless $|S| = 2$ is particularly useful for our proof.

In the proofs, we need tight lower bounds on $\bGi(s)$ for the following values of $s$.
\begin{lemma}
\label{lem:bGi_lower_bound}
For integers $s \ge 1$ and $m \ge 1$,
\[
  \bGi(s) \ge
  \begin{cases}
    4m    & s \ge 2m^{2}-2m, \\
    4m+1  & s \ge 2m^{2}-2m+2, \\
    4m+2  & s \ge 2m^{2}, \\
    4m+3  & s \ge 2m^{2}+1.
  \end{cases}
\]
\end{lemma}
\begin{proof}
Altshuler et al.~\cite[Theorem~1]{AltshulerYVWB06} showed that:
for every $m \ge 1$, $\bGi(2m^{2}) = 4m+2$ and $\bGi(2m^{2}+1) \ge 4m+3$; and 
for every $m \ge 2$, $\bGi(2m^{2}-2m) = 4m$ and $\bGi(2m^{2}-2m+2) = 4m+1$.
Since $\bGi(s)$ is non-decreasing~\cite[Lemma 2.1]{GuptaLMS21},
it suffices to check the cases where $m = 1$ and $s \ge 2m^{2}-2m$ or $s \ge 2m^{2}-2m+2$.
First, $m = 1$ and $s \ge 2m^{2}-2m$ only implies that $s \ge 0$, which is subsumed by the condition $s \ge 1$.
Thus, $\bGi(s) \ge \bGi(1) = 4 = 4m$ holds.
Next, $m = 1$ and $s \ge 2m^{2}-2m+2$ implies that $s \ge 2$.
Thus, we have $\bGi(s) \ge \bGi(2) = 6 = 4m+2 > 4m + 1$.
\end{proof}

The next lemma allows us to handle the infinite grid and a toroidal grid in the same way when a vertex set is separated from the rest by two rows and two columns.

\begin{lemma}
\label{lem:consecutive-empty-rows-and-columns}
If there are two consecutive rows and two consecutive columns that do not intersect $S \subseteq V(\Tn)$, then 
\[
  |\NTn(S)| 
  \ge
  \begin{cases}
    \bGi(|S|) & \text{for any } S, \\
    \bGi(|S|) + 1 & \text{if } S \text{ is disconnected and } |S| \ge 3.
  \end{cases}
\]
\end{lemma}
\begin{proof}
By symmetry, we may assume that the first row and the $n$th row do not intersect $S$. Let us call them $R_{1}$ and $R_{n}$.
Also, we may assume that the first column and the $n$th column do not intersect $S$, and we call them $C_{1}$ and $C_{n}$.

Let $\Tn^{-}$ be the $n \times n$ grid obtained from $\Tn$ by removing the wrap-around edges; 
i.e., the edges between $R_{1}$ and $R_{n}$ and the edges between $C_{1}$ and $C_{n}$.
Observe that $N_{\Tn^{-}}(S) = \NTn(S)$ since $S$ does not intersect $R_{1}$, $R_{n}$, $C_{1}$, and $C_{n}$.
Clearly, $\Tn^{-}$ is isomorphic to the induced subgraph $\Gi[[n]^{2}]$ of $\Gi$ 
with the identity mapping $(i,j) \mapsto (i,j)$ from $V(\Tn^{-})$ ($= [n]^{2}$) to $[n]^{2}$ as an isomorphism between them.
This implies that $N_{\Tn^{-}}(S) = \NGi(S)$, and thus it follows that
\[
  |\NTn(S)|
  =
  |N_{\Tn^{-}}(S)|
  =
  |\NGi(S)|
  \ge
  \begin{cases}
    \bGi(|S|) & \text{for any } S, \\
    \bGi(|S|) + 1 & \text{if } S \text{ is disconnected and } |S| \ge 3,
  \end{cases}
\]
where the disconnected case follows by the result of Gupta et al.~\cite[Theorem~G]{GuptaLMS21},
who showed that if a finite set $S \subseteq \Z^{2}$ with $|S| \ge 3$ is disconnected, then $|\NGi(S)| \ge \bGi(|S|) + 1$.
\end{proof}

In our proofs, we use \cref{lem:consecutive-empty-rows-and-columns} in the following forms.

\begin{lemma}
\label{lem:empty-row-and-column}
Let $S \subseteq V(\Tn)$ be a connected set. 
If there are two rows and two columns that do not intersect $S$, then $|\NTn(S)| \ge \bGi(|S|)$.
\end{lemma}
\begin{proof}
By \cref{lem:consecutive-empty-rows-and-columns}, it suffices to show that there are 
consecutive rows and columns that do not intersect $S$.
We show this claim only for rows below. The proof for columns is symmetric.

Let $R_{1}$ and $R_{2}$ be two rows that do not intersect $S$.
By symmetry, we may assume that $R_{1}$ is the first row and $R_{2}$ is the $i_{2}$th row ($2 \le i_{2} \le n$).
Since $S$ is connected, it lies \emph{between} $R_{1}$ and $R_{2}$, and thus, one of the following holds:
\begin{itemize}
  \setlength{\itemsep}{0pt}
  \item $1 < i < i_{2}$ for all $(i,j) \in S$;
  \item $i_{2} < i \le n$ for all $(i,j) \in S$.
\end{itemize}
In the first case, $S$ does not intersect the $n$th row (and the first row),
while in the second case, $S$ does not intersect the second row (and the first row).
\end{proof}

\begin{lemma}
\label{lem:almost-empty-row-and-column}
Let $S \subseteq V(\Tn)$. 
If there are a row $R$ and a column $C$
such that $S \cap (R \cup C) = \emptyset$, $|\NTn(S) \cap R| \le 1$, and $|\NTn(S) \cap C| \le 1$, then
\[
  |\NTn(S)| 
  \ge
  \begin{cases}
    \bGi(|S|) - 2 & \text{for any } S, \\
    \bGi(|S|) - 1 & \text{if } S \text{ is disconnected and } |S| \ge 3.
  \end{cases}
\]
\end{lemma}
\begin{proof}
\newcommand{\Tnp}{\torus{n+1}}
\newcommand{\NTnp}{N_{\Tnp}}

By symmetry, we may assume that $R$ is the first row and $C$ is the first column.
Let us consider the $(n+1) \times (n+1)$ toroidal grid $\Tnp$.
Since $S \subseteq V(\Tn) = [n]^{2} \subseteq [n+1]^{2} = V(\Tnp)$, the neighborhood $\NTnp(S)$ of $S$ in $\Tnp$ is defined as well.
In $\Tnp$, $S$ does not intersect the first and $(n+1)$st rows and the first and $(n+1)$st columns.
Hence, by \cref{lem:consecutive-empty-rows-and-columns} applied to $\Tnp$ and $S$, it suffices to show that $|\NTn(S)| \ge |\NTnp(S)| - 2$.

Observe that $\NTnp(S) \cap [n]^{2} \subseteq \NTn(S)$
since if two vertices in $[n]^{2}$ are adjacent in $\Tnp$, then they are adjacent in $\Tn$ as well.
Thus, $\NTnp(S) \setminus \NTn(S) = \NTnp(S) \cap (R_{n+1} \cup C_{n+1})$, where $R_{n+1}$ and $C_{n+1}$ are the $(n+1)$st row and column of $\Tnp$, respectively.
To complete the proof, it suffices to show that $|\NTnp(S) \cap R_{n+1}| \le 1$ and $|\NTnp(S) \cap C_{n+1}| \le 1$.
We only prove the former since the latter is symmetric.

Since $S$ does not intersect the first row of $\Tnp$,
the $n$th row of $\Tnp$ contains $|\NTnp(S) \cap R_{n+1}|$ vertices of $S$.
In $\Tn$, these $|\NTnp(S) \cap R_{n+1}|$ vertices in the intersection of $S$ and the $n$th row are adjacent to the same number of vertices in the first row $R$.
Therefore, it follows that $|\NTnp(S) \cap R_{n+1}| \le |\NTn(S) \cap R| \le 1$.
\end{proof}

\section{Proof for odd $n$}
\label{sec:odd}

In this section, we prove \eqref{eq:bramble_lb} for odd $n$.
That is, we show that $\bn(\Todd) \ge 4m+2$ for $m \ge 1$.
Note that we also cover the case of $\torus{3}$ here.

For each subset $J \subseteq V(\Todd)$ of size $4m+1$, we set $H_{J}$ to be a maximum component of $\Todd - J$.
If there are two or more maximum components, then we pick an arbitrary one of them as $H_{J}$.
We define a family $\mathcal{B} \subseteq 2^{V(\Todd)}$ as follows:
\[
  \mathcal{B} = \{V(H_{J}) \mid J \subseteq V(\Todd), \, |J| = 4m+1\}.
\]

We show that $\mathcal{B}$ is a bramble of $\Todd$ with order at least $4m+2$ for $m \ge 1$.
From the definition, each element of $\mathcal{B}$ is connected.
Also, no set $J \subseteq V(\Todd)$ of size $4m+1$ intersects all elements of $\mathcal{B}$ since $V(H_{J}) \subseteq V(\Todd) \setminus J$.
Hence, it suffices to show that any two elements of $\mathcal{B}$ touch.
To this end, we need two lemmas that roughly state the following properties:
\begin{itemize}
  \setlength{\itemsep}{0pt}
  \item each element of $\mathcal{B}$ is large (\cref{lem:lower_bound_bramble_odd});
  \item even a smallest element of $\mathcal{B}$ has large closed neighborhood (\cref{lem:lower_bound_neighbor_odd}).
\end{itemize}
Using these lemmas, we show that any two elements of $\mathcal{B}$ touch (\cref{lem:odd_touch}), completing the proof.

\begin{lemma}
\label{lem:lower_bound_bramble_odd}
Each element of $\mathcal{B}$ has size at least $2m^{2}$.
\end{lemma}
\begin{proof}
Let $J \subseteq V(\Todd)$ be a set with $|J| = 4m+1$.
It suffices to show that $\Todd - J$ contains a connected component with at least $2m^{2}$ vertices.
Observe that for each connected component $H$ of $\Todd - J$, $\NTodd(V(H)) \subseteq J$ holds.
Let $R$ and $C$ be a row and a column such that $|R \cap J| \le 1$ and $|C \cap J| \le 1$.
Such $R$ and $C$ exist since $|J| = 4m+1 < 2 \cdot (2m+1)$.

\proofsubparagraph{Case 1: There is a connected component of $\Todd - J$ that contains $(R \cup C) \setminus J$}
In this case, there is a connected component $H$ of $\Todd - J$ such that $(R \cup C) \setminus J \subseteq V(H)$.
Assume that $|V(H)| < 2m^{2}$ since otherwise we are done.
Let $Y = V(\Todd) \setminus (V(H) \cup J)$. Note that  $\NTodd(Y) \subseteq J$. 
We can see that $|Y| \ge 2m^{2} + 1$ as follows:
\[
  |Y| = |V(\Todd)| - |V(H)| - |J| \ge (2m+1)^{2} - (2m^{2}-1) - (4m+1) = 2m^{2} + 1.
\]
We assume that $Y$ is not connected, since otherwise we are done.
Observe that $(R\cup C) \cap Y = \emptyset$ since $R \cup C \subseteq V(H) \cup J = V(\Todd) \setminus Y$.
Observe also that $|\NTodd(Y) \cap R| \le |J \cap R| \le 1$ and $|\NTodd(Y) \cap C| \le |J \cap C| \le 1$.
Thus, \cref{lem:almost-empty-row-and-column} implies that $|\NTodd(Y)| \ge \bGi(|Y|) - 1$
as $Y$ is disconnected and $|Y| \ge 2m^{2}+1 \ge 3$.
Since $\bGi(|Y|) \ge 4m+3$ by \cref{lem:bGi_lower_bound} with $|Y| \ge 2m^{2}+1$,
it follows that $|\NTodd(Y)| \ge (4m+3) - 1 \ge 4m+2 > |J|$, contradicting $\NTodd(Y) \subseteq J$.

\proofsubparagraph{Case 2: There is no connected component of $\Todd - J$ that contains $(R \cup C) \setminus J$}
In this case, the unique vertex $v \in R \cap C$ has to belong to $J$.
Indeed, since $|R \cap J|, |C \cap J| \le 1$, we have $(R \cup C) \cap J = \{v\}$.
Observe that both $R\setminus \{v\}$ and $C \setminus \{v\}$ are connected sets in $\Todd - J$.
Thus, $J \setminus \{v\}$ has to separate $R\setminus \{v\}$ and $C \setminus \{v\}$.

Now we show that there are $4m$ internally vertex-disjoint $(R\setminus \{v\})$--$(C \setminus \{v\})$ paths that cover all vertices of $\Todd - (R \cup C)$.
To see this, consider a drawing of $\Todd$ such that $v$ is placed at the center (see \cref{fig:disjoint_paths}).
There are $m$ paths between $R \setminus \{v\}$ and $C \setminus \{v\}$ passing through the bottom-right $m \times m$ region:
for each $i \in [m]$, there is a shortest path $P_{i}$ from $v + (0,i)$ to $v + (i,0)$ via $v + (i,i)$;
that is, $P_{i} = (v + (0,i), v + (1,i), \dots, v+ (i,i), v+ (i,i-1), \dots, v + (i,0))$.
(Recall that we follow the matrix-like indexing for vertices.)
These $m$ paths are internally vertex-disjoint and pass through all vertices in the bottom-right region.
Since each of the other three regions contains $m$ such paths as well, in total, 
we have $4m$ internally vertex-disjoint $(R \setminus \{v\})$--$(C \setminus \{v\})$ paths that cover $V(\Todd) \setminus (R \cup C)$.
Let $\mathcal{P}$ be the set of these paths.

Since $|J \setminus \{v\}| = 4m$, each path in $\mathcal{P}$ contains exactly one vertex in $J \setminus \{v\}$.
This implies that $\Todd - J$ consists of exactly two connected components: one containing $R \setminus \{v\}$ and the other containing $C \setminus \{v\}$.
Since $|V(\Todd) \setminus J| = (2m+1)^{2} - (4m + 1) = 4m^{2}$, at least one of these components has $2m^{2}$ or more vertices.
\end{proof}
\begin{figure}[tbh]
  \centering
  \includegraphics{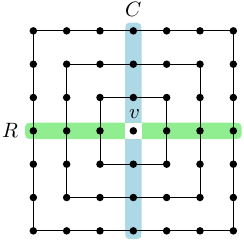}
  \caption{There are $4m$ internally vertex-disjoint paths between $R \setminus \{v\}$ and $C \setminus \{v\}$.}
  \label{fig:disjoint_paths}
\end{figure}

\begin{lemma}
\label{lem:lower_bound_neighbor_odd}
If $S \subseteq V(\Todd)$ is connected and $|S| = 2m^{2}$, then $|\NTodd(S)|\ge 4m+1$.
\end{lemma}
\begin{proof}
Let $r$ and $c$ be the numbers of rows and columns of $\Todd$ that $S$ intersects.
If $r \le 2m-1$ and $c \le 2m-1$, then \cref{lem:bGi_lower_bound,lem:empty-row-and-column} with $|S| = 2m^{2}$
imply that $|\NTodd(S)| \ge \bGi(|S|) \ge 4m+2$.

By symmetry, we may assume that $r \ge 2m$.
We first consider the case of $c \le 2m-1$. In this case, each row intersecting $S$ contains at least two elements of $\NTodd(S)$.
If $r = 2m+1$, then $|\NTodd(S)| \ge 2r = 4m+2$.
If $r = 2m$, then there is a unique row that does not intersect $S$, which contains an element of $\NTodd(S)$, and thus $|\NTodd(S)| \ge 2r+1 = 4m+1$.

In the remaining case, we have $r \ge 2m$ and $c \ge 2m$.
If each row intersecting $S$ contains at least two elements of $\NTodd(S)$,
then we can show that $|\NTodd(S)| \ge 4m+1$ as in the previous case.
Hence, we assume that there is a row $R$ that intersects $S$ and contains at most one element of $\NTodd(S)$.
Note that $R \subseteq \NTodd[S]$ since otherwise $R$ contains two or more elements of $\NTodd(S)$.
Similarly, we may assume that there is a column $C$ that intersects $S$ and contains at most one element of $\NTodd(S)$.
Again, we can see that $C \subseteq \NTodd[S]$.

Now, suppose to the contrary that $|\NTodd(S)|\le 4m$.
Let $X = V(\Todd) \setminus \NTodd[S]$. Note that $\NTodd(X) \subseteq \NTodd(S)$.
Since $|\NTodd(S)| \le 4m$,  it follows that $|X| = |V(\Todd) \setminus \NTodd[S]| = |V(\Todd)| - |S| - |\NTodd(S)| \ge (2m+1)^{2} - 2m^{2} - 4m = 2m^{2} + 1$.
By \cref{lem:bGi_lower_bound} with $|X| \ge 2m^{2}+1$, it follows that $\bGi(|X|) \ge 4m+3$.
Since $X$ does not intersect $R \cup C$ ($\subseteq \NTodd[S]$) and each of $R$ and $C$ contains at most one element of $\NTodd(X)$ ($\subseteq \NTodd(S)$),
we can apply \cref{lem:almost-empty-row-and-column} to $X$ and obtain $|\NTodd(X)| \ge \bGi(|X|) - 2 \ge 4m+1$.
This contradicts the assumption $|\NTodd(S)| \le 4m$ as $\NTodd(X) \subseteq \NTodd(S)$.
\end{proof}

\begin{lemma}
\label{lem:odd_touch}
Any two elements of $\mathcal{B}$ touch.
\end{lemma}
\begin{proof}
Let $X, Y \in \mathcal{B}$ with $X \ne Y$ and $|X| \ge |Y|$.
By \cref{lem:lower_bound_bramble_odd}, we have $|X|, |Y| \ge 2m^{2}$.

First assume that $|X| > 2m^{2}$.
Let $Z$ be a connected subset of $Y$ with $|Z| = 2m^{2}$.
\cref{lem:lower_bound_neighbor_odd} implies $|\NTodd(Z)| \ge 4m+1$.
Now we have
\[
  |X| + |\NTodd[Z]| 
  = |X| + |Z| + |\NTodd(Z)| 
  \ge 2m^{2} +1 + 2m^{2} + 4m+1  
  >
  |V(\Todd)|,
\]
which implies that $X$ intersects $\NTodd[Z]$.
Since $\NTodd[Z] \subseteq \NTodd[Y]$, $X$ touches $Y$.

Next assume that $|X| = 2m^{2}$ (and thus $|Y| = 2m^{2}$ as well).
Suppose to the contrary that $X$ and $Y$ do not touch.
Let $J = V(\Todd)\setminus (X\cup Y)$.
We have $|J| = 4m+1$ as $|V(\Todd)| = (2m+1)^{2}$.
Since $X$ and $Y$ do not touch, $\NTodd(X)$ is a subset of $J$.
Since $|X| = 2m^{2}$, \cref{lem:lower_bound_neighbor_odd} implies that $|\NTodd(X)| \ge 4m+1=|J|$.
Hence, $\NTodd(X) = J$. 
In the same way, it follows that $\NTodd(Y) = J$.
Since $|J| + |X| + |Y| = |V(\Todd)|$, the connected components of $\Todd - J$ are $\Todd[X]$ and $\Todd[Y]$.
On the other hand, since $J = \NTodd(X)$, the graph $\Todd[X]$ cannot be a connected component of $\Todd - J'$ for any other set $J' \ne J$ with $|J'| = 4m+1$.
The same holds for $\Todd[Y]$.
This contradicts the construction of $\mathcal{B}$, where we pick only one connected component of $\Todd - J$.
\end{proof}

\section{Proof for even $n$}
\label{sec:even}

In this section, we prove \eqref{eq:bramble_lb} for even $n$.
That is, we show that $\bn(\Teven) \ge 4m$ for $m \ge 3$.
Note that only \cref{lem:m-ball-properties,lem:even-bramble-touch} require $m \ge 3$ and the others hold for $m \ge 2$.

We first observe that the bramble construction for the odd case does not work for the even case.
That is, if we take the vertex set of a maximum component of $\Teven - J$ for each set $J$ of size $4m-1$,
then the resulting family contains non-touching sets.
\cref{fig:even_ce} shows such an example for $2m=6$.
In \cref{fig:even_ce}~(left) and (center), the crossed vertices form a set $J$ of $4m-1$ ($=11$) vertices
and the black vertices form the unique maximum component of $\Teven - J$.
\cref{fig:even_ce}~(right) shows that these maximum components do not touch.
We can see that this happens because of the set of the $4m-2$ ($=10$) white vertices in \cref{fig:even_ce}~(right)
that leaves two isomorphic $\ell_{1}$-balls of radius $m-1$ ($=2$) after its removal.
By removing one more vertex from one of the balls, we can make the other ball the unique maximum component.

\begin{figure}[tbh]
  \centering
  \includegraphics{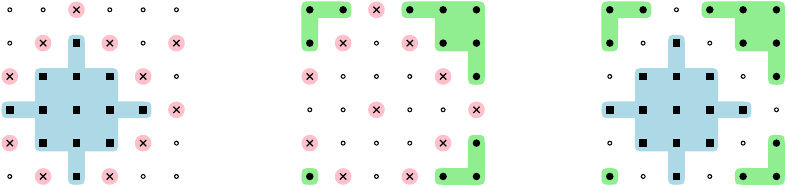}
  \caption{For $\Teven$, removing sets of $4m-1$ vertices may leave non-touching maximum components.}
  \label{fig:even_ce}
\end{figure}

In the following, we show that the example above essentially captures the \emph{unique} bad case.
This allows us to handle the exceptional case separately and define a bramble still similar to the odd case.

\subsection{The $m$-boundary sets and the $(m-1)$-balls}
\label{sec:even-m-boundary}
For a vertex $c \in V(\Teven)$, 
the \emph{$(m-1)$-ball centered at $c$} is the set $\{v \mid \dist_{\Teven}(c,v) \le m -1\}$ and
the \emph{$m$-boundary} of $c$ is the set $\{v \mid \dist_{\Teven}(c,v) = m\}$.
Note that if $B$ is the $(m-1)$-ball centered at $c$ and $M$ is the $m$-boundary of $c$, then $\NTeven(B) = M$.
If $M \subseteq V(\Teven)$ is the $m$-boundary of $c$ for some $c \in V(\Teven)$, we call $M$ an \emph{$m$-boundary set}.
Similarly, if $B \subseteq V(\Teven)$ is the $(m-1)$-ball centered at $c$ for some $c \in V(\Teven)$, we simply call $B$ an \emph{$(m-1)$-ball}.

We first count the number of vertices in an $(m-1)$-ball.

\begin{lemma}
\label{lem:m-1-ball-size}
An $(m-1)$-ball has size $2m^{2} - 2m + 1$.
\end{lemma}
\begin{proof}
Let $B$ be the $(m-1)$-ball centered at $c \in V(\Teven)$.
The set $B$ contains $c$ and, for each distance $d \in [m-1]$, the set of vertices
$\NTeven^{d}(c) \coloneqq \{v \mid \dist_{\Teven}(c,v) = d\} = \{c + (i,j) \mid i,j \in [-d,d], \; |i|+|j| = d\}$.
From the definition, $\NTeven^{d}(c) \cap \NTeven^{d'}(c) = \emptyset$ for distinct $d, d' \in [m-1]$.
Furthermore, for every $d \in [m-1]$, the $4d$ pairs $(i,j)$ with $|i|+|j| = d$ give rise to $4d$ distinct vertices $c + (i,j)$,
which can be verified by checking the extreme cases $c + (-(m-1),0) \ne c + (m-1,0)$ and $c + (0, -(m-1)) \ne c + (0,m-1)$.
Thus, it follows that $|B| = 1 + \sum_{d=1}^{m-1} 4d = 2m^{2} - 2m + 1$.
\end{proof}

The next lemma lists some properties of $m$-boundary sets. See \cref{fig:m-boundary}.
\begin{lemma}
\label{lem:m-boundary-properties}
Let $M \subseteq V(\Teven)$ be the $m$-boundary of $c \in V(\Teven)$. Then the following hold.
\begin{enumerate}
  \item $M = \{c + (i,j) \mid i,j \in [-m,m], \; |i|+|j| = m\}$.
  \label{itm:m-boundary-properties:def}

  \item $|M| = 4m-2$. 
  \label{itm:m-boundary-properties:size}

  \item $M$ is the $m$-boundary of $c + (m,m)$. 
  \label{itm:m-boundary-properties:another-center}

  \item $\Teven - M$ has exactly two connected components and each of them is induced by an $(m-1)$-ball.
  \label{itm:m-boundary-properties:two-balls}
\end{enumerate}
\end{lemma}
\begin{proof}
(\ref{itm:m-boundary-properties:def})
This is equivalent to the definition $M = \{v \mid \dist_{\Teven}(c,v) = m\}$
because a shortest path in $\Teven$ contains at most $m$ horizontal edges and at most $m$ vertical edges.

(\ref{itm:m-boundary-properties:size})
There are $4m$ combinations of $(i,j)$ satisfying $i,j \in [-m,m]$ and $|i|+|j| = m$,
and exactly two pairs of them represent the same vertices:
$c + (m,0) = c + (-m,0)$ and $c + (0,m) = c + (0,-m)$.

(\ref{itm:m-boundary-properties:another-center})
For $i,j \in [-m, m]$ with $|i| + |j| = m$,
$\dist_{\Teven}(c+(m,m), c+(i,j)) = \dist_{\Teven}((m,m), (i,j)) = m$.

(\ref{itm:m-boundary-properties:two-balls})
Let $B$ be the $(m-1)$-ball centered at $c$.
Since $B$ is connected and $\NTeven(B) = M$, $\Teven[B]$ is a connected component of $\Teven - M$.
Similarly, the $(m-1)$-ball $B'$ centered at $c + (m,m)$ induces a connected component of $\Teven - M$.
Since $\dist_{\Teven}(c, c+(m,m)) = 2m$, these two components are distinct.
Thus, by \cref{lem:m-1-ball-size}, it follows that $|M| + |B| + |B'| = (4m-2) + 2(2m^{2} - 2m + 1) = 4m^{2} = |V(\Teven)|$,
which implies that there is no other connected component of $\Teven - M$.
\end{proof}

\begin{figure}[tbh]
  \centering
  \includegraphics{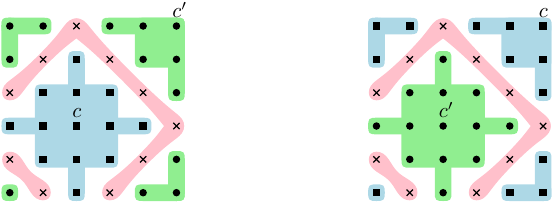}
  \caption{The $(m-1)$-balls centered at $c$ and $c' = c + (m,m)$ in two different drawings.}
  \label{fig:m-boundary}
\end{figure}

\begin{lemma}
\label{lem:m-ball-properties}
Let $m \ge 3$ and let $B$ be the $(m-1)$-ball centered at $c \in V(\Teven)$.
For every $v \in B$, the induced subgraph $\Teven[B \setminus \{v\}]$ has a unique maximum component $D$.
This component $D$ satisfies $|V(D)| \ge 2m^{2}-2m-1$ and $\NTeven(V(D)) = \NTeven(B) \cup \{v\}$.
\end{lemma}
\begin{proof}
Let $A = \{c + (i,j) \mid (i,j) \in \{(m-2,0), (-(m-2),0), (0,m-2), (0,-(m-2))\}\}$ be the set of neighbors of the degree-$1$ vertices in $\Teven[B]$.
We can see that for each $a \in A$, $\Teven[B \setminus \{a\}]$ has two components and they have size $1$ and $|B|-2$.
We show that there is no other cut vertex.
\begin{claim}
\label{clm:ball-cut}
$A$ is the set of cut vertices of $\Teven[B]$.
\end{claim}
\begin{subproof}[Proof of \cref{clm:ball-cut}]
For $d \in [m-1]$, let $H_{d}$ be the subgraph of $\Teven[B]$ induced by the vertices with distance at most $d$ from $c$,
and let $A_{d} = \{c + (i,j) \mid (i,j) \in \{(d-1,0), (-(d-1),0), (0,d-1), (0,-(d-1))\}\}$.
We show that for $d \in [m-1]$, every cut vertex of $H_{d}$ belongs to $A_{d}$.
This proves the claim since $H_{m-1} = \Teven[B]$ and $A_{m-1} = A$.

The case $d \le 2$ can be checked directly.
Assume that $d \ge 3$ and that $A_{d-1}$ is the set of cut vertices of $H_{d-1}$.
Let $v \in V(H_{d}) \setminus A_{d}$. We show that $H_{d} - \{v\}$ is connected. 

First consider the case where $v \notin A_{d-1}$, and thus $H_{d-1} - \{v\}$ is connected (no matter whether $v \in V(H_{d-1})$ or not).
Since $v \notin A_{d}$, each vertex of distance $d$ from $c$ has a neighbor in $V(H_{d-1}) \setminus \{v\}$, and thus $H_{d} - \{v\}$ is connected.

Next consider the case where $v \in A_{d-1}$.
By symmetry, we may assume that $v = c + (d-2,0)$.
Now, $H_{d-1} - \{v\}$ has two connected components with the vertex sets $\{v + (1,0)\}$ ($ = \{c + (d-1,0)\}$) and $H_{d-1} - \{v, v+(1,0)\}$.
It follows that $H_{d} - \{v\}$ is connected since each vertex of distance $d$ from $c$ has a neighbor in $V(H_{d-1}) \setminus \{v\}$
and there is a path $(v+(1,0))$--$(v+(1,1))$--$(v+(0,1))$ from $v+(1,0)$ to $V(H_{d-1}) \setminus \{v, v+(1,0)\}$.
\end{subproof}

By \cref{lem:m-1-ball-size}, $|B| = 2m^{2}-2m+1$.
By \cref{clm:ball-cut}, a maximum component, say $D$, has at least $2m^{2}-2m-1$ vertices.
This component is unique since $2|V(D)| > |B|$.
To complete the proof of the lemma, we need to show that $\NTeven(V(D)) = \NTeven(B) \cup \{v\}$.

We first show that $\NTeven(B) \subseteq \NTeven(V(D))$.
We have $v \in \NTeven(V(D))$ since $v \in B$ and $D$ is a connected component of $\Teven[B \setminus \{v\}]$.
Thus, to prove that $\NTeven(B) \subseteq \NTeven(V(D))$, it suffices to show that each vertex in $\NTeven(B)$ has
at least two neighbors in $\{u \mid \dist_{\Teven}(c,u) = m-1\}$ ($\subseteq B$)
since $D$ misses at most one vertex of $\{u \mid \dist_{\Teven}(c,u) = m-1\}$.
Recall that $\NTeven(B)$ is the $m$-boundary of $c$; that is, $\NTeven(B) = \{u \mid \dist_{\Teven}(c,u) = m\}$.
If a vertex $u \in \NTeven(B)$ is $c + (m,0)$ ($=c + (-m,0)$) or $c + (0,m)$ ($=c + (0,-m)$),
then $u$ has two vertical neighbors $c + (m-1,0)$ and $c + (-m+1,0)$ or two horizontal neighbors $c + (0,m-1)$ and $c + (0,-m+1)$, respectively.
Otherwise $u \in \NTeven(B) \setminus \{c + (m,0), c + (0,m)\}$, and thus it has one vertical neighbor and one horizontal neighbor in $\{w \mid \dist_{\Teven}(c,w) = m-1\}$. 

Finally, we show that $\NTeven(V(D)) \subseteq \NTeven(B) \cup \{v\}$.
Indeed, this is an immediate corollary of the facts that
$D$ is a connected component of $\Teven[B \setminus \{v\}]$ and that
$\NTeven(V(D)) \subseteq  \NTeven(B \setminus \{v\}) = \NTeven(B) \cup \{v\}$.
\end{proof}

The next lemma is the key lemma for the even case
that says that the $m$-boundary sets are the only obstacles preventing the odd-case construction here.
(Note that this lemma holds for $m \ge 2$.)
\begin{lemma}
\label{lem:only_m-boundary_for_two_2m^2-2m}
Let $m \ge 2$ and $M \subseteq V(\Teven)$ be a set with $|M| \le 4m-2$.
If $\Teven - M$ has two connected components each having at least $2m^{2} - 2m$ vertices,
then $M$ is an $m$-boundary set.
\end{lemma}
\begin{proof}
Let $X$ and $Y$ be the vertex sets of two connected components of $\Teven - M$ such that $|X|, |Y| \ge 2m^{2} - 2m$.
Since $\NTeven(X), \NTeven(Y) \subseteq M$, it follows that $|\NTeven(X)|, |\NTeven(Y)| \le |M| \le 4m-2$.
Since there are $2m$ rows, either 
there is a row not intersecting $M$, or
there are two rows such that each of them contains exactly one vertex of $M$.
The same holds for columns as well.

\begin{claim}
\label{clm:only_m-boundary:M_intersect}
All rows and columns intersect $M$.
\end{claim}
\begin{subproof}[Proof of \cref{clm:only_m-boundary:M_intersect}]
Suppose to the contrary that  
there is a row or a column that does not intersect $M$.
By symmetry, we may assume that there is a row $R$ not intersecting $M$.
Since $R$ is connected, at least one of the connected sets $X$ and $Y$ is disjoint from $R$.
We may assume by symmetry that $X \cap R = \emptyset$.
Observe that $X$ does not intersect the rows next to $R$ since $\NTeven(X) \subseteq M$ and $M \cap R = \emptyset$.
Thus, there are at least two (indeed, three including $R$) rows not intersecting $X$.

If there is a column $C$ not intersecting $M$, then $X \cap (R \cup C) = \emptyset$ since $R \cup C$ is connected and $X \cap R = \emptyset$.
Thus, as before, there are two columns not intersecting $X$.
Otherwise, there are two columns $C_{1}$ and $C_{2}$ such that $|C_{1} \cap M| = |C_{2} \cap M| = 1$.
Observe that $(R \cup C_{1} \cup C_{2}) \setminus M$ is connected because, for $i \in \{1,2\}$, 
removing one vertex from $C_{i} \setminus R$ cannot separate $C_{i}$ from $R$.
This implies that $X \cap ((R \cup C_{1} \cup C_{2}) \setminus M) = X \cap (R \cup C_{1} \cup C_{2}) = \emptyset$, and thus, $X \cap (C_{1} \cup C_{2}) = \emptyset$.
In both cases, we have two columns that do not intersect $X$.

By \cref{lem:empty-row-and-column}, it follows that $|\NTeven(X)| \ge \bGi(|X|)$.
Now, \cref{lem:bGi_lower_bound} with $|X| \ge 2m^{2}-2m$ implies that $|\NTeven(X)| \ge \bGi(|X|) \ge 4m > |M|$, which contradicts $\NTeven(X) \subseteq M$.
\end{subproof}

By \cref{clm:only_m-boundary:M_intersect}, there are two rows $R_{1}$, $R_{2}$ and two columns $C_{1}$, $C_{2}$ 
such that $|R_{1} \cap M| = |R_{2} \cap M| = |C_{1} \cap M| = |C_{2} \cap M| = 1$.
If $X$ is disjoint from $R_{1} \cup R_{2} \cup C_{1} \cup C_{2}$,
then $|\NTeven(X)| \ge 4m$ holds by \cref{lem:bGi_lower_bound,lem:empty-row-and-column}, a contradiction.
Thus, $X$ intersects $R_{1} \cup R_{2} \cup C_{1} \cup C_{2}$.
By the same argument, it follows that $Y$ intersects $R_{1} \cup R_{2} \cup C_{1} \cup C_{2}$ as well.
Since $X$ and $Y$ are connected, $(R_{1} \cup R_{2} \cup C_{1} \cup C_{2}) \setminus M$ has to be disconnected.
\begin{claim}
\label{clm:only_m-boundary:R-C-intersection}
Let $R_{i} \cap M = \{v_{i}\}$ for $i \in \{1,2\}$.
Then, $v_{i} \in C_{1} \cup C_{2}$ for $i \in \{1,2\}$.
\end{claim}
\begin{subproof}[Proof of \cref{clm:only_m-boundary:R-C-intersection}]
By symmetry, it suffices to show that $v_{1} \in C_{1} \cup C_{2}$.
Suppose to the contrary that $v_{1} \notin C_{1} \cup C_{2}$.
Observe that $R_{1} \setminus M$, $R_{2} \setminus M$, $C_{1} \setminus M$, and $C_{2} \setminus M$ are connected as they induce paths.
The assumption $v_{1} \notin C_{1} \cup C_{2}$ implies that 
$R_{1} \setminus M$ and $C_{i} \setminus M$ share a vertex in $R_{1}$ for $i \in \{1,2\}$.
Hence, $(R_{1} \cup C_{1} \cup C_{2}) \setminus M$ is connected.
Since $R_{2} \setminus M$ includes at least one vertex of $C_{1} \setminus M$ or $C_{2} \setminus M$,
we obtain a contradiction that $(R_{1} \cup R_{2} \cup C_{1} \cup C_{2}) \setminus M$ is connected. 
\end{subproof}

Now, by swapping $C_{1}$ and $C_{2}$ if necessary, we assume that $v_{1} \in C_{2}$ and $v_{2} \in C_{1}$.
That is, $R_{1} \cap M = C_{2} \cap M = R_{1} \cap C_{2} \cap M = \{v_{1}\}$
and $R_{2} \cap M = C_{1} \cap M = R_{2} \cap C_{1} \cap M = \{v_{2}\}$.
This implies that $(R_{i} \cap C_{i}) \setminus M \ne \emptyset$ for $i \in \{1,2\}$, and thus $(R_{i} \cup C_{i}) \setminus M$ is connected.
See \cref{fig:only_m-boundary}.

\begin{claim}
\label{clm:only_m-boundary:non-consecutive}
$R_{1}$, $R_{2}$ and $C_{1}$, $C_{2}$ are non-consecutive.
\end{claim}
\begin{subproof}[Proof of \cref{clm:only_m-boundary:non-consecutive}]
Suppose to the contrary that $R_{1}$ and $R_{2}$ are consecutive. (The case of columns is symmetric.)
Since $2m > 2$, there is a column that intersects both $R_{1} \setminus M$ and $R_{2} \setminus M$.
Hence, $(R_{1} \cup R_{2}) \setminus M$ is connected.
Since $(R_{i} \cup C_{i}) \setminus M$ is connected for each $i \in \{1,2\}$, 
it follows that $(R_{1} \cup R_{2} \cup C_{1} \cup C_{2}) \setminus M$ is connected, a contradiction.
\end{subproof}

By arranging $\Teven$ so that the intersection of $R_{2}$ and $C_{2}$ is placed at the top-right corner, we define the number $a$ and $b$ such that
$a$ is the number of rows between $R_{1}$ and $R_{2}$ counted from top to bottom and $b$ is the number of columns between $C_{1}$ and $C_{2}$ counted from left to right.
See \cref{fig:only_m-boundary}.
\begin{figure}[tbh]
  \centering
  \includegraphics{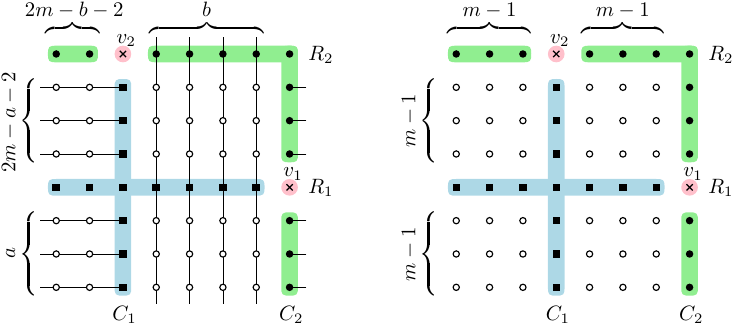}
  \caption{Internally vertex-disjoint paths between $(R_{1} \cup C_{1}) \setminus M$ and $(R_{2} \cup C_{2}) \setminus M$.}
  \label{fig:only_m-boundary}
\end{figure}

\begin{claim}
\label{clm:only_m-boundary:a=b=m-1}
$a = b = m-1$.
\end{claim}
\begin{subproof}[Proof of \cref{clm:only_m-boundary:a=b=m-1}]
From \cref{fig:only_m-boundary}, it follows that
in the four rectangular regions separated by $R_{1}$, $R_{2}$, $C_{1}$, and $C_{2}$,
there are $\max\{a,b\}$, $\max\{a, 2m-b-2\}$, $\max\{b,2m-a-2\}$, and $\max\{2m-a-2,2m-b-2\}$ internally vertex-disjoint paths 
between $(R_{1} \cup C_{1}) \setminus M$ and $(R_{2} \cup C_{2}) \setminus M$.
Since $M$ separates $R_{1} \cup C_{1}$ and $R_{2} \cup C_{2}$, it has to intersect all those paths.
This gives us the following lower bound of $|M|$:
\begin{align*}
  |M| 
  &\ge 
  2 + \max\{a,b\}+ \max\{2m-a-2,2m-b-2\} + \max\{a, 2m-b-2\} + \max\{b,2m-a-2\}
  \\
  &=
  2 + \max\{a-(m-1),b-(m-1)\}+ \max\{(m-1)-a, (m-1)-b\} \\
  &\qquad + \max\{a-(m-1), (m-1)-b\} + \max\{b-(m-1), (m-1)-a\} + 4(m-1)
  \\&=
  4m-2 + |(a-(m-1)) - (b-(m-1)) | + |(a-(m-1)) - ((m-1)-b)|
  \\&=
  4m-2 + |a - b| + |a+b-2(m-1)|,
\end{align*}
where the second line is obtained by subtracting $m-1$ from each $\max\{\cdot\}$ term and then canceling them by adding $4(m-1)$,
and the third line follows since $\max\{x,y\}+\max\{-x,-y\}=\max\{x-y,y-x\}=|x-y|$ for any integers $x$ and $y$.
Since $|M| \le 4m-2$,
it follows that $a = b$ and $a+b = 2(m-1)$, and thus $a = b = m-1$.
\end{subproof}

Now we determine exactly which vertices belong to $M$. 
For $i \in \{1,2\}$, let $B_{i}$ be the vertex set of the connected component of $\Teven - M$ that contains $(R_{i} \cup C_{i}) \setminus M$.
Observe that in each of the four $(m-1) \times (m-1)$ square regions separated by $R_{1}$, $R_{2}$, $C_{1}$, and $C_{2}$,
the set $M$ contains exactly one vertex in each row and in each column to hit all paths between $B_{1}$ and $B_{2}$ while keeping $|M| \le 4m-2 = 4(m-1)+2$,
where the $+2$ term counts $v_{1}$ and $v_{2}$.
This implies that each row or column other than $R_{1}$, $R_{2}$, $C_{1}$, and $C_{2}$ contains exactly two elements of $M$.

Let $c$ be the unique vertex in $R_{1} \cap C_{1}$.
We show that $M$ is the $m$-boundary of $c$.
Recall that $v_{1}, v_{2} \in M$.
It follows that $v_{1} = c + (0,m)$ and $v_{2} = c + (m,0)$.
Now consider the four vertices 
\begin{align*}
&v_{2} + (+1, +1) = c + (m+1, +1) = c + (-(m-1), +1), \\
&v_{2} + (+1, -1) = c + (m+1, -1) = c + (-(m-1), -1), \\
&v_{2} + (-1, +1) = c + (m-1, +1), \\
&v_{2} + (-1, -1) = c + (m-1, -1).
\end{align*}
These four vertices are adjacent to both $B_{1}$ and $B_{2}$, and thus they belong to $M$ (see \cref{fig:only_m-boundary_M}~(left)).
Each of the rows next to $R_{2}$ and the columns next to $C_{1}$ contains two of these four elements of $M$,
and hence the other vertices in these rows and columns  belong to $B_{1}$ or $B_{2}$ depending on which set they are adjacent to (see \cref{fig:only_m-boundary_M}~(center)).

We repeat this process as many times as possible.
In general, at the $i$th step, 
we find that the following four vertices are adjacent to both $B_{1}$ and $B_{2}$, and thus they belong to $M$ (see \cref{fig:only_m-boundary_M}~(center and right) for $i \in \{2,3\}$):
\begin{align*}
&v_{2} + (+i, +i) = c + (m+i, +i) = c + (-(m-i), +i), \\
&v_{2} + (+i, -i) = c + (m+i, -i) = c + (-(m-i), -i), \\
&v_{2} + (-i, +i) = c + (m-i, +i), \\
&v_{2} + (-i, -i) = c + (m-i, -i).
\end{align*}
Then the other vertices in the same rows and columns with one of these four vertices can be placed into $B_{1}$ or $B_{2}$ 
depending on their adjacency to already revealed vertices of $B_{1}$ and $B_{2}$.

After the $(m-1)$st step, we stop and find a partition of $V(\Teven)$ into $M$, $B_{1}$, and $B_{2}$.
With $v_{1} = c + (0,m)$ ($= c+(0,-m)$) and $v_{2} = c + (m,0)$ ($=c+(-m,0)$), 
it follows that $M = \{c + (i,j) \mid i,j \in [-m,m], \; |i|+|j| = m\}$.
Therefore, $M$ is the $m$-boundary of $c$.
\end{proof}
\begin{figure}[tbh]
  \centering
  \includegraphics{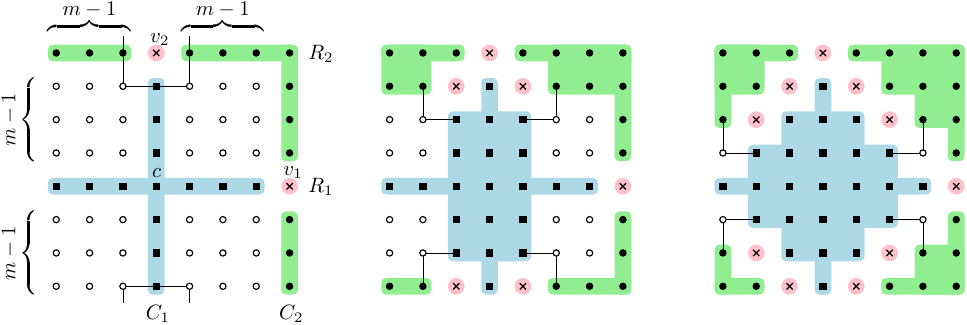}
  \caption{A small separator $M$ has to be an $m$-boundary set.}
  \label{fig:only_m-boundary_M}
\end{figure}

\subsection{Constructing a bramble of $\Teven$}

Now we are ready to construct a bramble of $\Teven$ with order $4m$ for $2m \ge 6$.
We say that a set $J \subseteq V(\Teven)$ is \emph{$m$-boundary-free} if it contains no $m$-boundary set as a subset.
\begin{itemize}
  \setlength{\itemsep}{0pt}

  \item For each $m$-boundary-free set $J \subseteq V(\Teven)$ of size $4m-1$,
  we set $H_{J}$ to be a maximum component of $\Teven - J$.
  If there are two or more maximum components, then we pick an arbitrary one of them as $H_{J}$.

  \item Let $M \subseteq V(\Teven)$ be an $m$-boundary set.
  By \cref{lem:m-boundary-properties}~(\ref{itm:m-boundary-properties:two-balls}),
  $\Teven - M$ has exactly two connected components and each of them is induced by an $(m-1)$-ball.
  Arbitrarily choose one of the balls, and call it $B_{M}$. (We call the other ball $B_{M}'$.)
  For each $v \in B_{M}$, let $B_{M}^{v}$ be the vertex set of the unique maximum component of $\Teven[B_{M} \setminus \{v\}]$
  (it is unique by \cref{lem:m-ball-properties}).
\end{itemize}
We set $\mathcal{E} = \mathcal{E}_{1} \cup \mathcal{E}_{2}$, where $\mathcal{E}_{1}$ and $\mathcal{E}_{2}$ are defined as follows:
\begin{align*}
  \mathcal{E}_{1} &= \{V(H_{J}) \mid m\text{-boundary-free set } J \subseteq V(\Teven), \, |J| = 4m-1\},
  \\
  \mathcal{E}_{2} &= \{B_{M}^{v} \mid m\text{-boundary set } M \subseteq V(\Teven), \, v \in B_{M}\}.
\end{align*}

We show that $\mathcal{E}$ is a bramble of $\Teven$ with order at least $4m$.
From the definition, each element of $\mathcal{E}$ is connected.
The next observation shows that a hitting set of $\mathcal{E}$ has size at least $4m$ as follows.

\begin{observation}
\label{obs:even-sets-order}
No set $S \subseteq V(\Teven)$ of size $4m-1$ intersects all elements of $\mathcal{E}$.
\end{observation}
\begin{proof}
If $S$ is an $m$-boundary-free set, then $V(H_{S}) \in \mathcal{E}_{1}$ and $V(H_{S}) \cap S = \emptyset$ as $H_{S}$ is a connected component of $\Teven - S$.
Otherwise, $S = M \cup \{v\}$ for some $m$-boundary set $M$ and some vertex $v \notin M$.
If $v \in B_{M}$, then $B_{M}^{v} \in \mathcal{E}_{2}$ and $B_{M}^{v} \cap S = \emptyset$
as $B_{M}^{v}$ induces a connected component of $\Teven - S$.
If $v \notin B_{M}$, then $S = M \cup \{v\}$ does not intersect $B_{M}^{w} \in \mathcal{E}_{2}$ for any $w \in B_{M}$
since $B_{M}$ induces a connected component of $\Teven - M$, and thus
$B_{M}^{w} \cap S \subseteq B_{M} \cap S = B_{M} \cap (M \cup \{v\}) = \emptyset$.
\end{proof}

Now, it suffices to show that any two elements of $\mathcal{E}$ touch.
We first show that each element of $\mathcal{E}$ is large
and then we proceed to the final step of the proof that puts everything together.

\begin{lemma}
\label{lem:even_lower_bound_2m^2-2m}
Let $m \ge 2$.
If $J \subseteq V(\Teven)$ has size $4m-1$,
then a maximum component of $\Teven - J$ has at least $2m^{2}-2m$ vertices.
\end{lemma}
\begin{proof}
Since $|J| = 4m-1$, there are a row $R$ and a column $C$ such that $|R \cap J| \le 1$ and $|C \cap J| \le 1$.

\proofsubparagraph{Case 1: There is a connected component of $\Teven - J$ that contains $(R \cup C) \setminus J$}
Let $X$ be the vertex set of the connected component of $\Teven - J$ such that $(R \cup C) \setminus J \subseteq X$.
Assume that $|X| \le 2m^{2}-2m-1$ since otherwise we are done.
Let $Y = V(\Teven) \setminus (X \cup J)$. It follows that $|Y| \ge (2m)^{2} -(2m^{2}-2m-1)-(4m-1) = 2m^{2}-2m+2$,
and thus, if $Y$ is connected, we are done.
Assume that $Y$ is disconnected.
Since $(R \cup C) \cap Y = \emptyset$, $\NTeven(Y)\subseteq J$, $|R \cap J| \le 1$, $|C \cap J| \le 1$, and $|Y| \ge 2m^{2}-2m+2 \ge 3$,
\cref{lem:almost-empty-row-and-column} implies that $|\NTeven(Y)| \ge \bGi(|Y|) - 1$.
By \cref{lem:bGi_lower_bound} with $|Y| \ge 2m^{2}-2m+2$,
we have $\bGi(|Y|) \ge 4m+1$.
This implies $|\NTeven(Y)| \ge (4m+1) - 1 = 4m > |J|$, a contradiction.

\proofsubparagraph{Case 2: There is no connected component of $\Teven - J$ that contains $(R \cup C) \setminus J$}
Let $R \cap C = \{v\}$.
By the assumptions, $R \cap J = C \cap J = \{v\}$.
Let $X$ and $Y$ be the vertex sets of the connected components of $\Teven - J$
that contain $R \setminus \{v\}$ and $C \setminus \{v\}$, respectively.
The assumption implies that $X \ne Y$.

If there is another row $R' \ne R$ that satisfies $|R' \cap J| \le 1$, then $(R' \cup C) \setminus J$ is connected,
and thus, we are done by the previous case.
Thus, we assume that each row other than $R$ includes at least two vertices of $J$.
Indeed, since $|J \setminus \{v\}| = 4m-2$, each of these $2m-1$ rows has exactly two vertices of $J$.
By the same argument for columns, it follows that each column other than $C$ includes exactly two vertices of $J$.

Now we show that the current case gives rise to a contradiction that $|J| < 4m-1$.
The proof here is similar to the last paragraphs of the proof of \cref{lem:only_m-boundary_for_two_2m^2-2m}.
See \cref{fig:only_m-boundary_J}. Consider the four vertices $v+(1,1)$, $v+(1,-1)$, $v+(-1,1)$, $v+(-1,-1)$.
All of them are adjacent to both $X$ and $Y$, and thus they belong to $J$.
Now the rows and columns next to $R$ and $C$ already have two elements of $J$ and hence cannot have any other elements of $J$.
The remaining vertices in those rows and columns belong to $X$ or $Y$ depending on which set they are adjacent to.
Then, it follows that the four vertices $v+(2,2)$, $v+(2,-2)$, $v+(-2,2)$, $v+(-2,-2)$ are adjacent to both $X$ and $Y$ and thus belong to $J$.
We repeat this process $m$ times and then after all we obtain
\[
  J = \{v\} \cup \{v+(d,d), \, v+(d,-d), \, v+(-d,d), \, v+(-d,-d) \mid d \in [m]\}.
\]
Since $v+(m,m) = v+(-m,-m) = v+(m,-m) = v+(-m,m)$, it follows that $|J| = 4m-2 < 4m-1$, a contradiction.
(Note that this set $J$ is the $m$-boundary of $v + (m,0)$.)
\end{proof}
\begin{figure}[tbh]
  \centering
  \includegraphics[width=\textwidth]{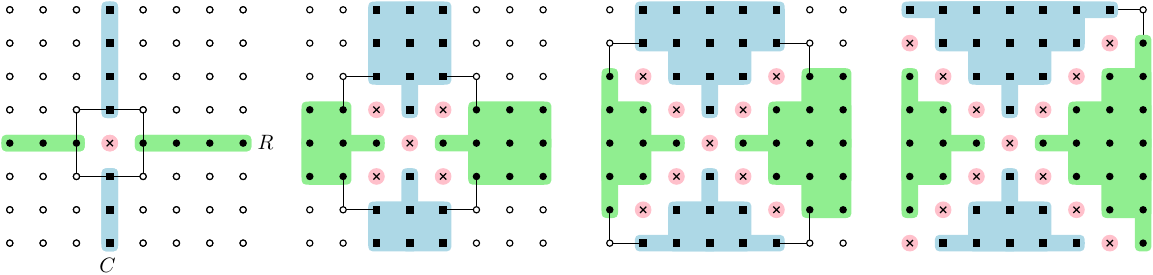}
  \caption{$J = \{v, v+(m,m)\} \cup \{v+(d,d), \, v+(d,-d), \, v+(-d,d), \, v+(-d,-d) \mid d \in [m-1]\}$.}
  \label{fig:only_m-boundary_J}
\end{figure}

By \cref{lem:m-ball-properties,lem:even_lower_bound_2m^2-2m}, we have the following lower bounds on the size of elements of $\mathcal{E}$.
\begin{observation}
\label{obs:even-bramble-element-size-lb}
Each element of $\mathcal{E}_{1}$ has size at least $2m^{2}-2m$
and
each element of $\mathcal{E}_{2}$ has size at least $2m^{2}-2m-1$.
\end{observation}

Now we are ready to complete the proof.

\begin{lemma}
\label{lem:even-bramble-touch}
Any two elements of $\mathcal{E}$ touch.
\end{lemma}
\begin{proof}
Let $X, Y \in \mathcal{E}$. Suppose to the contrary that $X$ and $Y$ do not touch.

\proofsubparagraph{Case 1: At least one of $X$ and $Y$ belongs to $\mathcal{E}_{2}$}
By symmetry, we may assume that $X \in \mathcal{E}_{2}$.
From the definition of $\mathcal{E}_{2}$, $X = B_{M}^{v}$ such that $M$ is an $m$-boundary set and $v \in B_{M}$.
Since $m \ge 3$, \cref{lem:m-ball-properties} implies that $\NTeven(X) = M \cup \{v\}$.
Since $X$ and $Y$ do not touch, $Y \subseteq V(\Teven) \setminus \NTeven[X]$. 
Since $|Y| \ge 2m^{2}-2m-1$ ($> 1$) by \cref{obs:even-bramble-element-size-lb}
and 
\begin{align*}
  |V(\Teven) \setminus (\NTeven[X] \cup B_{M}')| 
  &= 
  |V(\Teven)| - |\NTeven[X]| - |B_{M}'| 
  = 
  |V(\Teven)| - |X| - |M \cup \{v\}| - |B_{M}'| 
  \\
  &\le 
  4m^{2} - (2m^{2}-2m-1) - (4m-1) - (2m^{2}-2m+1) 
  = 1,
\end{align*}
it follows that $Y \subseteq B_{M}'$ as $Y$ is connected.

We first show that $Y \in \mathcal{E}_{2}$.
Suppose to the contrary that $Y \in \mathcal{E}_{1}$, and thus there is a set $J$ such that $Y$ is the vertex set of a maximum component of $\Teven - J$.
By \cref{obs:even-bramble-element-size-lb}, we have $|Y| \ge 2m^{2}-2m$.
Since $B_{M}'$ is an $(m-1)$-ball, $|B_{M}' \setminus Y| \le 1$ by \cref{lem:m-1-ball-size}.
Hence, $Y$ is a connected subgraph of $\Teven[B_{M}']$ with at most one missing vertex.
This implies, with \cref{lem:m-ball-properties}, that $M \subseteq \NTeven(Y) \subseteq J$, contradicting that $J$ is $m$-boundary-free.

Since $Y \in \mathcal{E}_{2}$, there is an $m$-boundary set $L$ such that $Y \subseteq B_{L}$.
If $L = M$, then $B_{L} \cap B_{M}' = \emptyset$, contradicting that $Y \subseteq B_{L} \cap B_{M}'$. Thus, $L \ne M$.
Let $c$ and $c'$ be the centers of $B_{L}$ and $B_{M}'$, respectively.
Assume without loss of generality that $c' = c + (i,j)$ for $0 \le i \le j \le m$ with $1 \le i+j < 2m$.
\begin{itemize}
  \setlength{\itemsep}{0pt}
  \item If $(i,j) = (0,1)$, then $\{c + (m-1,0), c + (-(m-1),0), c + (0,-(m-1))\} \subseteq B_{L} \setminus B_{M}'$.
  \item If $(i,j) = (1,1)$, then $\{c + (-(m-2),-1), c + (-(m-1),0), c + (0,-(m-1))\} \subseteq B_{L} \setminus B_{M}'$.
  \item Otherwise $j \ge 2$, and thus
  $\{c + (-1, -(m-j)), c + (0, -(m-j)), c + (1, -(m-j)) \} \subseteq B_{L} \setminus B_{M}'$.
\end{itemize}
It follows that $|Y| \le |B_{L} \cap B_{M}'| \le |B_{L}| - 3 = (2m^{2}-2m+1) -3 < 2m^{2}-2m -1$,
a contradiction to \cref{obs:even-bramble-element-size-lb}.

\proofsubparagraph{Case 2: $X, Y \in \mathcal{E}_{1}$}
By the definition of $\mathcal{E}_{1}$, there are $m$-boundary-free sets $J_{X}$ and $J_{Y}$ of size $4m-1$ such that
$X$ and $Y$ are the vertex sets of maximum components of $\Teven - J_{X}$ and $\Teven - J_{Y}$, respectively.
We have $|\NTeven(X)| \le 4m-1$ and $|\NTeven(Y)| \le 4m-1$ since $\NTeven(X) \subseteq J_{X}$ and $\NTeven(Y) \subseteq J_{Y}$.

We now show that $|\NTeven(X)| = |\NTeven(Y)| = 4m - 1$.
\cref{obs:even-bramble-element-size-lb} implies $|X|, |Y| \ge 2m^{2}-2m$.
Since $X$ and $Y$ do not touch, $Y$ is a subset of the vertex set $Y'$ of some connected component of $\Teven - \NTeven(X)$.
If $|\NTeven(X)| \le 4m - 2$, then \cref{lem:only_m-boundary_for_two_2m^2-2m} implies that $\NTeven(X)$ is an $m$-boundary set.
This contradicts the assumptions that $J_{X}$ ($\supseteq \NTeven(X)$) is $m$-boundary-free.
Thus, we have $|\NTeven(X)| = 4m - 1$.
By the same argument, we can show that $|\NTeven(Y)| = 4m - 1$.

We next show that indeed $\NTeven(X) = \NTeven(Y)$.
If $Y$ is the vertex set of a connected component of $\Teven - \NTeven(X)$,
then $\NTeven(Y) \subseteq \NTeven(X)$ holds, and thus $\NTeven(X) = \NTeven(Y)$ as $|\NTeven(X)| = |\NTeven(Y)|$.
Assume that $Y$ is not the vertex set of any connected component of $\Teven - \NTeven(X)$.
This implies that there is a connected component of $\Teven - \NTeven(X)$ whose vertex set, say $Y'$, properly contains~$Y$.
Now we have $|Y'| \ge |Y| + 1 \ge 2m^{2}-2m+1$.
Since $X$ is the vertex set of a maximum component of $\Teven - \NTeven(X)$,
it holds that $|X| \ge |Y'| \ge 2m^{2}-2m+1$. This gives a contradiction, since
\[
  |Y'| + |X| + |\NTeven(X)| \ge 2 \cdot (2m^{2}-2m+1) + (4m-1) = 4m^{2} +1 > |V(\Teven)|.
\]

Since $\NTeven(X) \subseteq J_{X}$ and both sets have size $4m-1$, we have $J_{X}=\NTeven(X)$; similarly $J_{Y}=\NTeven(Y)$. 
To complete the proof, observe that $J_{X}$ ($= J_{Y}$) is the unique set of size $4m-1$ such that
$X$ and $Y$ appear as the vertex sets of connected components of $\Teven - \NTeven(X)$ ($=\Teven - \NTeven(Y)$).
This contradicts the definition of $\mathcal{E}_{1}$, which includes the vertex set of at most one connected component of $\Teven - J$ for each vertex set $J$ of size $4m-1$.
\end{proof}

\section{Concluding remarks}
\label{sec:conclusion}

In this paper, we showed that the treewidth of the $n \times n$ toroidal grid is $2n-1$ for all $n \ge 5$
by constructing brambles using an approach different from those in previous studies.
Applying our approach to other graphs with unknown treewidth might be useful in some cases.
For example, our approach may be useful to tackle the remaining case $\tw(\torus{n,n+1}) \in \{2n-1,2n\}$~\cite{AidunDMYY20}.

Note that our bramble for $\Tn$ has size potentially exponential in~$n$,
while the previous ones for $\Tn$ have size polynomial in~$n$~\cite{Wood13,KiyomiOO16,AidunDMYY20}.
It is known that some graphs admit only super-polynomial size brambles of maximum order~\cite{GroheM09}.
It would be interesting to ask if $\Tn$ admits a maximum-order bramble of polynomial size.
For example, it is well known that the $n \times n$ grid admits a maximum-order bramble with $\Theta(n^{2})$ elements (see~\cite{BodlaenderGK08}). 

\bibliographystyle{plainurl}
\bibliography{ref}

\end{document}